\documentclass[reqno]{amsart}
\usepackage[utf8]{inputenc}
\usepackage{amssymb,amsmath,bm,enumerate,stmaryrd,amsthm,tikz-cd, leftidx}
\usepackage{stackrel}
\usetikzlibrary{matrix}
\usetikzlibrary{arrows}
\usepackage[all]{xy}
\usepackage[shortlabels]{enumitem}
\usepackage{colonequals}

\usepackage[colorlinks]{hyperref}
\hypersetup{
 colorlinks=true,
 linkcolor=blue,
 filecolor=magenta, 
 urlcolor=cyan,
}
\usepackage[capitalise]{cleveref}
\crefformat{equation}{(#2#1#3)}
\crefrangeformat{equation}{(#3#1#4--#5#2#6)}
\crefformat{enumi}{(#2#1#3)}
\crefrangeformat{enumi}{(#3#1#4--#5#2#6)}

\usepackage[pagewise]{lineno}
\let\oldequation\equation
\let\oldendequation\endequation
\renewenvironment{equation}{\linenomathNonumbers\oldequation}{\oldendequation\endlinenomath}
\expandafter\let\expandafter\oldequationstar\csname equation*\endcsname
\expandafter\let\expandafter\oldendequationstar\csname endequation*\endcsname
\renewenvironment{equation*}{\linenomathNonumbers\oldequationstar}{\oldendequationstar\endlinenomath}
\let\oldalign\align
\let\oldendalign\endalign
\renewenvironment{align}{\linenomathNonumbers\oldalign}{\oldendalign\endlinenomath}
\expandafter\let\expandafter\oldalignstar\csname align*\endcsname
\expandafter\let\expandafter\oldendalignstar\csname endalign*\endcsname
\renewenvironment{align*}{\linenomathNonumbers\oldalignstar}{\oldendalignstar\endlinenomath}

\usepackage{tikz-cd}
\usetikzlibrary{decorations.pathmorphing}

\makeatletter
\newcommand{\xra}{\xrightarrow}

\newcommand{\ges}{\geqslant}
 \newcommand{\les}{\leqslant}
\newcommand{\del}{\partial}

\newcommand{\D}{\mathsf{D}}
\newcommand{\Dfp}{\mathsf{D^f_+}}

\renewcommand{\S}{\mathcal{S}}
\newcommand{\m}{\mathfrak{m}}

\newcommand{\f}{\bm{f}}

\newcommand{\vp}{{\varphi}}

\DeclareMathOperator{\ann}{ann}

\DeclareMathOperator{\Ext}{Ext}

\newcommand{\pdim}{\operatorname{proj\,dim}}
\newcommand{\grade}{\operatorname{grade}}
\DeclareMathOperator{\Hom}{Hom}
\DeclareMathOperator{\F}{F}

\newcommand{\E}{\operatorname{\leftidx{^2}{E}{}}}

\newcommand{\Ker}{\operatorname{Ker}}

\newcommand{\lotimes}{\otimes^{\operatorname{L}}}
\DeclareMathOperator{\rank}{rank}
\newcommand{\RHom}{\operatorname{RHom}}
\newcommand{\shift}{{\sf\Sigma}}

\DeclareMathOperator{\Tor}{Tor}
\DeclareMathOperator{\depth}{depth}
\DeclareMathOperator{\cx}{cx}
\DeclareMathOperator{\curv}{curv}
\DeclareMathOperator{\edim}{edim}
\DeclareMathOperator{\ps}{P}
\DeclareMathOperator{\h}{H}

\newcommand{\coloneq}{\colonequals}
\renewcommand{\P}{\mathrm{P}}

\makeatother

\newcounter{intro}
\newtheorem{introthm}[intro]{Theorem}

\newtheorem{theorem}[subsection]{Theorem}
\newtheorem{proposition}[subsection]{Proposition}
\newtheorem{lemma}[subsection]{Lemma}
\newtheorem{corollary}[subsection]{Corollary}

\theoremstyle{definition}

\newtheorem{chunk}[subsection]{}
\newtheorem*{ack}{Acknowledgements}

\theoremstyle{remark}
\newtheorem{remark}[subsection]{Remark}

\newtheorem{notation}[subsection]{Notation}

\numberwithin{equation}{subsection}

\title[quasi-complete intersection maps]{Relations between Poincar\'{e} series for \\ quasi-complete intersection homomorphisms}

\author[J.~Pollitz]{Josh Pollitz}
\address{
Mathematics Department, 
Syracuse University, 
Syracuse, NY 13244 U.S.A.}
\email{jhpollit@syr.edu}

\author[L.~\c{S}ega]{Liana \c{S}ega}
\address{
Division of Computing, Analytics and Mathematics, University of Missouri, Kansas City,  MO 64110, U.S.A.}
\email{segal@umkc.edu}

\date{\today}

\keywords{quasi-complete intersection homomorphism, Koszul complex, dg algebra, dg module, Poincar\'e series, large homomorphisms, inertness}
\subjclass[2020]{Primary: 13D40. Secondary: 13D02, 13D07, 16E45.}

\begin{document}

\begin{abstract}
In this article we study base change of Poincar\'e series along a quasi-complete intersection homomorphism $\varphi\colon Q \to R$, where $Q$ is a local ring with maximal ideal $\mathfrak{m}$. In particular, we give a precise relationship between the Poincar\'e series $\mathrm{P}^Q_M(t)$ of a finitely generated $R$-module $M$ to $\mathrm{P}^R_M(t)$ when the kernel of $\varphi$ is contained in $\mathfrak{m}\,\mathrm{ann}_Q(M)$. This generalizes a classical result of Shamash for complete intersection homomorphisms. Our proof goes through base change formulas for Poincar\'e series under the map of dg algebras $Q\to E$, with $E$ the Koszul complex on a minimal set of generators for the kernel of $\varphi.$
\end{abstract}

\maketitle

%\tableofcontents

%%%%%%%%%%%%%%%%%%%%%%%%%%%%%%%%%%%%%%
\section*{Introduction}
\label{sec:introduction}
%%%%%%%%%%%%%%%%%%%%%%%%%%%%%%%%%%%%%%

This article is concerned with change of base formulas for Poincar\'e series in commutative algebra. Recall the Poincar\'e series of a finitely generated module, over a local ring, is the generating series of its sequence of Betti numbers. For a complete intersection homomorphism, this problem has been extensively studied and the relationship between the Poincar\'e series over the source and target is well understood; see, for example, \cite{Avramov/Gasharov/Peeva:1997,Nagata:1962,Shamash:1969}. In this article we study this problem for the much larger class of homomorphisms called \emph{quasi-complete intersection} (abbreviated to q.c.i.) homomorphisms. These homomorphisms are precisely the ones that satisfy the conclusion of a long-standing conjecture of Quillen~\cite{Quillen:1970}, and have been a topic of much recent research~\cite{Avramov/Henriques/Sega:2013,Bergh/Celikbas/Jorgensen:2014,Blanco/Majadas/Rodicio:1998,DeBellevue/Pollitz:2023,Garcia/Soto:2004,Kustin/Sega:2023,Soto:2000}. 

Let $\vp\colon Q\to R$ be a surjective local homomorphism, and let $E$ denote the Koszul complex on a minimal set of generators for $I=\Ker\vp$. If the homology of $E$ is isomorphic to the exterior algebra (over $R$) on $\h_1(E)$ and $\h_1(E)$ is free over $R$, then $\vp$ is said to be q.c.i. Such a homomorphism can equivalently be defined in terms of admitting a two-step Tate resolution; see \cref{c:qci}. In \cite{Avramov/Henriques/Sega:2013}, the authors  investigated  such homomorphisms and gave relationships between the Poincar\'e series $\ps^Q_M$  and $\ps^R_M$ of finitely generated $R$-modules $M$ over $Q$ and $R$. More precisely, when $M=k$ or the minimal generators of $I$ can be extended to minimal generators of the maximal ideal $\m$ of $Q$, they proved the formula:
  \begin{equation}\label{e:qciformula1}
    \P_M^R(t)\cdot\frac{(1-t)^{\edim R}}{(1-t^2)^{\depth R}}=\P_M^Q(t)\cdot\frac{(1-t)^{\edim Q}}{(1-t^2)^{\depth Q}}\,.
    \end{equation}
    In particular, when $\varphi$ is q.c.i., formula \eqref{e:qciformula1} holds precisely when $M$ is inert by $\varphi$, in the sense of Lescot ~\cite{Lescot:1990}.
The aforementioned formula generalizes results of Tate~\cite{Tate:1957} and Nagata~\cite{Nagata:1962} that hold when $\vp$ is complete intersection (meaning that $I$ is generated by a regular sequence). Formula \eqref{e:qciformula1} is also known in the complete intersection case when $I\subseteq \m \ann_Q(M)$, due to Shamash~\cite{Shamash:1969}; the authors in \cite{Avramov/Henriques/Sega:2013} comment that it is not known if Shamash's result can be extended to the q.c.i.\@ case. Our main result establishes this extension: 

\begin{introthm}
\label{intro:thm}
    Let $(Q,\m,k)$ be a local ring, $\vp\colon Q\to R$ a surjective quasi-complete intersection map, and  set $I=\Ker\vp$. Then  a finitely generated $R$-module $M$ with $I\subseteq \m\ann_Q(M)$ is inert by $\vp$; equivalently, $M$ satisfies \eqref{e:qciformula1}.
\end{introthm}

This is \cref{t:main} in the paper and is presented with a slightly different (but equivalent) formulation, cf.\@ \cref{r:equivalent-inert}. We also show in \cref{e:generalineq} that there is a more general inequality that holds for any q.c.i.\@ map $\vp\colon Q\to R=Q/I$; namely, if $n$ and $m$ denote the minimal number of generators of $I$ and its first Koszul homology, respectively, then  for any finitely generated $R$-module $M$ we have
\begin{equation*}
\P^Q_M(t)\preccurlyeq \frac{(1+t)^{n-m}}{(1-t)^m}\P_M^R(t)\,,
\end{equation*}
with equality whenever $I\cap \m^2\subseteq \m I$. 

The proofs of these results are given after first establishing, in \cref{sec:koszul}, intermediary results that describe $\ps^Q_M$ in terms of the Poincar\'e series of $M$ regarded as a differential graded (abbreviated to \emph{dg}) module over $E$. Here $E$ is viewed as a dg $Q$-algebra in the usual way, i.e.~an exterior algebra on $E_1$ with differential equal to the unique $Q$-linear derivation determined by mapping a basis of $E_1$ bijectively to a minimal generating set for $I$.  The main result from \cref{sec:koszul}, applied to prove aforementioned results in \cref{sec:qci}, is the following:
\begin{introthm}
\label{intro:thm2}
Fix a local ring $(Q,\m,k)$, an ideal  $I$ of $Q$ minimally generated by a sequence of length $n$, and set $E$ to be the Koszul complex on a minimal generating set of $I$ and $R=Q/I$. For each bounded below complex of finitely generated $R$-modules $M$, there are coefficient-wise inequalities:
\begin{align*}
\ps_M^E(t)&\preccurlyeq \ps_M^Q(t)\cdot (1-t^2)^{-n}\,;\\
\ps_M^Q(t)&\preccurlyeq \ps_M^E(t)\cdot (1+t)^n\,.
\end{align*}
Furthermore,
\begin{enumerate}%[\quad\rm(1)]
\item if $I \subseteq \m\ann_Q(M)$, then equality holds in the first inequality above;
\item if $I\cap \m^2\subseteq \m I$, then equality holds in the second inequality above.
\end{enumerate}
\end{introthm}

The equalities in \cref{intro:thm2} generalize the known results for complete intersection homomorphisms mentioned above after \cref{e:qciformula1} to arbitrary surjective maps; the only catch is that one must replace the local ring $R$ with the dg $Q$-algebra $E$, which is quasi-isomorphic to $R$ only when $\vp$ is complete intersection. The idea of replacing $R$ by $E$ to witness complete intersection-like behavior is one previously exploited in \cite{Pollitz:2019,Pollitz:2021}; it is worth highlighting that the numerical results in this article are a new utility of this perspective.

%%%%%%%%%%%%%%%%%%%%%%%%%%%%%%%%%%%%%%
\section{Background}
\label{sec:background}
%%%%%%%%%%%%%%%%%%%%%%%%%%%%%%%%%%%%%%

Throughout, $Q$ will denote a commutative noetherian local ring with maximal ideal $\m$ and residue field $k$. Recall that a differential graded, henceforth \emph{dg}, $Q$-algebra is a graded $Q$-algebra equipped with compatible differential. That is to say, a graded $Q$-algebra $A=\{A_i\}_{i\in \mathbb{Z}}$ with a degree $-1$ endomorphism $\del$ satisfying  $\del^2=0$ and the Leibniz rule:
\[
\del(a\cdot b)=\del(a)\cdot b+(-1)^{|a|}a\cdot \del(b)\,;
\]
here $|a|$ denotes the unique value $i$ for which $a$ belongs to $A_i.$ The reader is directed to \cite{Avramov:2010} for the necessary background on dg algebras. 

\begin{chunk}
We say a dg $Q$-algebra $A$ is local if it is non-negatively graded, $(A_0,\m_0)$ is a local ring, and each $\h_i(A)$ is finitely generated over $\h_0(A).$ In this case, we write $\m_A$ for the maximal dg ideal  of $A$; explicitly, 
\[
\m_A=\m_{0}\oplus A_1\oplus A_2\oplus \cdots \,.
\]
\end{chunk}

For the remainder of the section, fix a local dg $Q$-algebra $A$ whose residue field $A/\m_A$ is $k$, the residue field of $Q$. 

\begin{chunk}
Let $\D(A)$ denote the derived category of dg $A$-modules; cf.\@ \cite{Krause:2022}  or \cite[Section~2]{Avramov/Buchweitz/Iyengar/Miller:2010}. We write $\Dfp(A)$ for the full subcategory of $\D(A)$ consisting of all dg $A$-modules $M$ where $\h_i(M)=0$ for $i\ll 0$ and each $\h_i(M)$ is finitely generated over $\h_0(A)$.  

We let $(-)^\natural$ denote the functor that forgets the differential of a dg $A$-module and regards it as a graded module. That is to say, if $M$ is a dg $A$-module, then $M^\natural$ is the underlying graded module over the graded algebra $A^\natural$. 
\end{chunk}

\begin{chunk}
\label{c:resolution}
Next, we recount some background on semifree dg modules; cf.\@ \cite[Section~1]{Avramov/Halperin:1986} (see also \cite[Chapter~6]{Felix/Halerpin/Thomas:2001}).  Recall a dg $A$-module $F$ is semifree if admits an exhaustive filtration by dg $A$-submodules
\[
0=F(-1)\subseteq F(0)\subseteq F(1)\subseteq \ldots \subseteq F
\]
where each subquotient $F(i)/F(i-1)$ is a direct sum of shifts of $A$. In the present setting, any bounded below dg $A$-module $F$ with $F^\natural$ a free graded $A^\natural$-module is semifree. For every dg $A$-module $M$ there exists a semifree dg $A$-module $F$ and a quasi-isomorphism $F\to M$, that is unique up to homotopy equivalence; we call such a map (or the semifree module) a \emph{semifree resolution} of $M$ over $A$. 

If $M$ is in $\Dfp(A)$, then there exists a semifree resolution $F$ of $M$ over $A$ satisfying: 
\begin{enumerate}
    \item $\del^F(F)\subseteq \m_A F\,,$ and
    \item $F^\natural \cong \oplus_{i\in \mathbb{Z}}\shift^i(A^{\beta_i})^\natural$ where $\beta_i=0$ for $i\ll 0;$
\end{enumerate}
see, for example, \cite[Appendix~B.2]{Avramov/Iyengar/Nasseh/SatherWagstaff:2019}.
We will refer to such a resolution $F$ as a \emph{minimal} semifree resolution of $M$ over $A$. 
\end{chunk}

\begin{chunk}
\label{c:exttor}
A dg $A$-module $M$ defines exact endofunctors $-\lotimes_A M$ and $\RHom_A(M,-)$ on $\D(A)$ given by $-\otimes_A F$ and $\Hom_A(F,-)$, respectively, where $F\to M$ is a semifree resolution of $M$ over $A$.
These are well-defined by \cref{c:resolution}. Set 
\[
\Ext_A(M,-)\colonequals \h(\RHom_A(M,-))\quad \text{and}\quad \Tor^A(M,-)\colonequals \h(M\lotimes_A -)\,.
\]
\end{chunk}

\begin{chunk}
\label{c:poincare}
Let $M$ be in $\Dfp(A)$.
The $i^{\text{th}}$ \emph{Betti number of $M$ over $A$} is 
\[
\beta^A_i(M)\coloneq \rank_k \Tor^A_i(M,k)=\rank_k\Ext_A^i(M,k)\,.
\]
These are finite for all $i$ and zero for $i\ll 0$;
see \cref{c:resolution}.  
The \emph{Poincar\'{e} series of $M$ over $A$} is the formal Laurent series 
\[
\ps^A_M(t)\coloneq \sum_{i\in\mathbb{Z}}\beta^A_i(M) \  t^i \,.
\]
For a graded $k$-vector space $V=\{V_i\}_{i\in \mathbb{Z}}$, we write $\h_V(t)$ for its Hilbert series 
\[
\h_V(t)=\sum_{i\in \mathbb{Z}}(\rank_k V_i) \ t^i\,.
\]
So if $F\to M$ is a minimal semifree resolution of $M$ over $A$, then  $\h_{F\otimes_A k}(t)=\P^A_M(t).$
\end{chunk}

\begin{chunk}
\label{c:gamma}
We write $A\langle X\rangle$ for the semifree dg $A$-algebra extension obtained by successively adjoining variables to kill cycles in the sense of Tate; see \cite{Tate:1957} (as well as \cite[Section~6]{Avramov:2010} or \cite{Gulliksen/Levin:1969}). Here $X=X_1,X_2,\ldots $ where each $X_i$ consists of exterior variables when $i$ is odd and divided power variables when $i$ is even.  Hence, as a graded $A$-algebra, $A\langle X\rangle$ is the free strictly graded-commutative divided power algebra over $A$ on $X$. 

When $A$ is a ring and $\f=f_1,\ldots,f_n$ is a sequence of elements in $A$, then adjoining the degree one variables $X=\{e_1,\ldots,e_n\}$ to kill the cycles  $\f$ produces
\[A\langle X_1\rangle= A\langle e_1,\ldots,e_n\mid \del(x_i)=f_i\rangle \] 
 the Koszul complex on $\f$ over $A$. Note that $\h_0(A\langle X_1\rangle )=A/(\f)$, and hence each dg $A/(\f)$-module is a dg $A\langle X_1\rangle$-module via restriction of scalars along the augmentation $A\langle X_1\rangle\to A/(\f)$. In particular, for any $A/(\f)$-complex $M$, we have 
\begin{equation}
    \label{restricting_e_action}
    e_iM=0\quad\text{for each}\quad i=1,\ldots,n\,.
\end{equation}
\end{chunk}

\begin{chunk}\label{c:change_of_rings} 
We now adapt to our dg setting the classical Cartan--Eilenberg change of ring spectral sequence~\cite[Chapter XVI, Section 5]{Cartan/Eilenberg:1956}. We provide the details for this extension to the (slightly) more general setting needed in what follows. 
Given a map of non-negatively graded dg algebras $A\to B$, and  $M$, $N$ bounded below dg $B$ modules,  there is a spectral sequence 
\[
\E_{p,q}=\Tor^B_p(M,\Tor^A_q(B,N))\implies \Tor^A_{p+q}(M,N)
\]
with differentials 
\[
^r\mathrm{d}_{p,q}\colon ^r\mathrm{E}_{p,q}\to \ ^r\mathrm{E}_{p-r,q+r-1}
\]
constructed as follows. 

Let $V\to M$ be a semifree resolution of $M$ over $B$ and $W\to N$ a semifree resolution of $N$ over $A$. By \cite[Proposition 1.3.2]{Avramov:2010}, the induced map $V\otimes_AW\to M\otimes_AW$ is a quasi-isomorphism, and so we  make identifications
$$\Tor^A(M,N)=\h(V\otimes_AW)=\h(V\otimes_B(B\otimes_AW))\,.$$ 
Let $V_{\leqslant p}$ be the semifree dg $B$-submodule of $V$ with $(V_{\les p})^\natural$ the free graded $B^\natural$-module generated by the basis element of $V^\natural$ in homological degrees at most $p$. 
The filtration of $V$ by these sub dg $B$-modules induces a filtration 
\[
\F_p(V\otimes_B(B\otimes_AW))=\text{Im} \left(V_{\leqslant p}\otimes_B(B\otimes_AW)\to V\otimes_B(B\otimes_AW)\right)\,.
\]
The spectral sequence obtained from this filtration is 
\[
\E_{p,q}=\h_p(V\otimes_B\h_q(B\otimes_AW))\implies \Tor^A_{p+q}(M,N)
\]
with differentials as above; here the $B$-action on $\h_q(B\otimes_A W)$ is through the augmentation $B\to \h_0(B)$. We further identify 
\[
\h_q(B\otimes_AW)\!=\!\Tor_q^A(B,N)\quad\text{and}\quad \h_p(V\otimes_B\Tor_q^A(B,N))\!=\!\Tor^B_p(M,\Tor^A_q(B,N)).
\]

To justify convergence, we can forget the algebra structures and regard $V\otimes_A W$ as a complex filtered by the subcomplexes $\F_p(V\otimes_B(B\otimes_AW)).$
Note that for each  integer $n$, the filtration of \[(V\otimes_A W)_n=\left(V\otimes_B(B\otimes_AW)\right)_n\] 
by its submodules $\left(\F_p(V\otimes_B(B\otimes_AW))\right)_n$ is finite, because $V$ and $W$ are bounded below. The filtration on $V\otimes_AW$ is thus bounded. Using \cite[10.14]{Rotman:2009},  this implies that the spectral sequence converges to $V\otimes_AW$, in the sense that for each $n\in \mathbb Z$, the module $\h_n(V\otimes_AW)$ has a bounded filtration such that for each $q$ the  component of degree $q$ of the associated graded module is
isomorphic to $^\infty\mathrm{E}_{n-q,q}$. 
\end{chunk}

\begin{lemma}
\label{l:large-ineq}
If $\varphi\colon A\to B$ is a map of local dg algebras with residue field $k$ and $M$ is in $\Dfp(B)$, then there is a coefficient-wise inequality of Poincar\'e series 
\begin{equation}
\label{e:large-ineq}
\ps^A_M(t)\preccurlyeq \ps^B_M(t)\cdot \ps^A_B(t)\,.
\end{equation}
\end{lemma}

\begin{proof} Consider the spectral sequence in \ref{c:change_of_rings} with $N=k$. For all $p$, $q$ we have isomorphisms of  $k$-vector spaces 
\[
\Tor^B_p(M,\Tor^A_q(B,k))\cong \Tor^B_p(M,k)\otimes_k \Tor^A_q(B,k)
\]
which yield 
\[
\rank_k \E_{p,q}=\beta_p^B(M)\beta_q^A(B)\,.
\]
A rank count in the spectral sequence then gives
\[
\ps^A_M(t)=\sum_{i\in \mathbb Z}\beta^A_n(M)t^i\preccurlyeq \sum_{i\in \mathbb Z}\left(\sum_{p+q=i}\rank_k \E_{p,q}\right)t^i= \P_M^B(t)\cdot P_B^A(t)\,.\qedhere
\]
\end{proof}

% \begin{lemma}
%     Let $A\to B$ be a map of local dg algebras with common residue field $k$. For a dg $A$-module $M$, there is a coefficientwise inequality on Poincaré series:
%     \[
%     \ps^A_M(t)\preccurlyeq \ps^B_M(t)\cdot \ps^A_B(t)\,.
%     \]
% \end{lemma}
% \begin{proof}
% Fix a minimal semifree resolution $F\to M$ a semifree resolution of $M$ over $B$, and a minimal semifree  resolution $G\to B$ of $B$ over $A$. 
% \[
% (F\otimes_B k)\otimes_k (k\otimes_A G)\cong F\otimes_B (k\otimes_A G)\simeq F\otimes_B (L\otimes_A G)\
% \]
% \end{proof}

\begin{remark}
\label{r:large}
   When $A$ and $B$ are local rings, Levin \cite{Levin:1980} shows that equality holds in \eqref{e:large-ineq} for all finitely generated $B$-modules $M$ if and only if the induced homomorphism $\Tor^\vp(k,k)\colon \Tor^A(k,k)\to \Tor^B(k,k)$ is surjective. There Levin called homomorphisms satisfying this property {\it large}. We adopt Levin's terminology and say that a map $\varphi\colon A\to B$ of local dg algebras augmented to $k$ is large if equality holds in \eqref{e:large-ineq} for all $M$ in $\Dfp(B)$. Levin's proof of \cite[Theorem~1.1]{Levin:1980} carries through to show that $\varphi$ is large if and only if the induced map on Tor algebras $\Tor^\vp(k,k)\colon \Tor^A(k,k)\to \Tor^B(k,k)$ is surjective. 
\end{remark}

%%%%%%%%%%%%%%%%%%%%%%%%%%%%%%%%%%%%%%
\section{Poincar\'e series over the Koszul complex}
\label{sec:koszul}
%%%%%%%%%%%%%%%%%%%%%%%%%%%%%%%%%%%%%%
Continuing with notation from \cref{sec:background}, $Q$ is a commutative noetherian local ring with maximal ideal $\m$ and residue field $k$. 
When $Q\to R$ is a surjective  homomorphism of local rings, with kernel generated by a regular sequence, and $M$ is a finitely generated $R$-module, (in)equalities between $\ps_M^R(t)$ and $\ps_M^Q(t)$ are well-known, and are recalled in   \cref{c:recall-ci} below. In this section we show that these results have dg versions that hold without assuming that the kernel is generated by a regular sequence, see \cref{t:Koszul-ineq-eq}, which recover the classical results.

\begin{chunk}
\label{c:recall-ci}
Fix a local ring $(Q,\m,k)$ and set $R=Q/I$ where  $I$ is an ideal generated by a regular sequence of length $n$. Recall that for each finitely generated $R$-module $M$, there are coefficient-wise inequalities: 
\begin{align}
\label{e:RQ}\ps_M^R(t)&\preccurlyeq \ps_M^Q(t)\cdot (1-t^2)^{-n}\,;\\
\label{e:QR}\ps_M^Q(t)&\preccurlyeq \ps_M^R(t)\cdot (1+t)^n\,.
\end{align}
Furthermore,
\begin{enumerate}%[\quad\rm(1)]
\item \label{shamash_ci} if $I\subseteq \m\ann_Q(M)$, then equality holds in \eqref{e:RQ};
\item \label{nagata_ci} if $I\cap \m^2\subseteq \m I$, then equality holds in \eqref{e:QR}.
\end{enumerate}
In \cite[Theorem~5]{Tate:1957}, Tate showed \cref{shamash_ci} when $M$ is cyclic, and the general result is due to Shamash~\cite[Corollary~1, Section~3]{Shamash:1969}, who also provides a proof of \cref{nagata_ci} in \cite[Corollary~1, Section~2]{Shamash:1969}. The original proof of \cref{nagata_ci} is implicit in work of Nagata~\cite[Section~27]{Nagata:1962} where it is expressed in terms of ranks of syzygies. A more modern, and comprehensive, treatment of these (in)equalities is contained in \cite[Section~3.3]{Avramov:2010}. 
\end{chunk}

The main result of the section is the following. 

\begin{theorem} 
\label{t:Koszul-ineq-eq}
Fix a local ring $(Q,\m,k)$, an ideal  $I$ of $Q$ minimally generated by a sequence of length $n$, and set $E$ to be the Koszul complex on a minimal generating set of $I$ and $R=Q/I$. For each $M$ in $\Dfp(R)$, there are coefficient-wise inequalities:
\begin{align}
\label{e:RQ2}\ps_M^E(t)&\preccurlyeq \ps_M^Q(t)\cdot (1-t^2)^{-n}\,;\\
\label{e:QR2}\ps_M^Q(t)&\preccurlyeq \ps_M^E(t)\cdot (1+t)^n\,.
\end{align}
Furthermore,
\begin{enumerate}%[\quad\rm(1)]
\item \label{t_shamash_ci} if $I \subseteq \m\ann_Q(M)$, then equality holds in \eqref{e:RQ2};
\item \label{t_nagata_ci }if $I\cap \m^2\subseteq \m I$, then equality holds in \eqref{e:QR2}.
\end{enumerate}
\end{theorem}

The proof \cref{t:Koszul-ineq-eq} will be given at the end of the section, after we introduce the needed ingredients. 

\begin{notation} 
\label{notation}
For the rest of the section we fix an ideal $I$ of $Q$ and a minimal generating set $\f=f_1,\ldots, f_n$ of $I$ and  we let 
\[
E= Q\langle e_1,\ldots, e_n\mid \del e_i=f_i\rangle
\]
be the Koszul complex on $\f$ over $Q$; cf.\@ \cref{c:gamma}. 

Set $\S=Q[\chi_1,\ldots,\chi_n]$ where $\chi_i$ has degree $-2$; this can be identified with the ring of cohomology operators introduced by Eisenbud~\cite{Eisenbud:1980} and Gulliksen~\cite{Gulliksen:1974}; cf.\@ \cite{Avramov/Sun:1998}.  Define
\[
\varGamma\coloneq Q\langle y_1,\ldots, y_n\rangle\,,
\] the free divided power algebra on degree two divided power variables $y_1,\ldots,y_n.$  It is well-known that $\varGamma$ can be naturally identified with the graded $Q$-linear dual of $\S$ and hence it admits the structure of a graded $\S$-module; this is a classical structure introduced by Macaulay in \cite{Macaulay}. Namely, a graded $Q$-basis for $\varGamma$ is given by $\{y^{(H)}\coloneq y_1^{(h_1)}\cdots y_n^{(h_n)}\mid H=(h_1,\ldots,h_n)\in \mathbb{N}^n\}$ and the $\S$-action is determined by
\[
\chi_i\cdot y^{(H)}=\begin{cases}
    y^{(h_1,\ldots, h_{i-1},h_i-1,h_{i+1},\ldots,h_n)}& h_i\geqslant 1\\
    0 & \text{otherwise}\,.
\end{cases}
\]
\end{notation}

\begin{chunk}{\bf A semifree resolution over $E$.}
\label{c:kresolutions}
Let $M$ be a dg $E$-module and fix a semifree resolution $\epsilon\colon F\xra{\simeq} M$ of $M$ over $Q$, where $F$ is a dg $E$-module and $\epsilon$ is a homomorphism of dg $E$-modules. Such a semifree resolution exists by \cite[2.1]{Avramov/Buchweitz:2000a}; this result does not use the assumption that $\f$ is a (Koszul-)regular sequence which was present in Section 2 of \cite[Section~2]{Avramov/Buchweitz:2000a}. By  \cite[4.2.2]{Pollitz:2021}, which is essentially due to \cite[Proposition~2.6]{Avramov/Buchweitz:2000a}, the semifree dg
$E$-module
\begin{align*}
U_E(F)&\coloneq E\otimes_Q \varGamma\otimes_Q F \quad\text{with differential}\\ \del&\coloneq\del^E\otimes 1\otimes 1+1\otimes 1\otimes \del^F+\sum_{i=1}^ne_i\otimes \chi_i\otimes 1-1\otimes \chi_i\otimes e_i
\end{align*}
and augmentation 
\[
U_E(F)\to M\quad \text{given by}\quad e\otimes y\otimes x\mapsto\begin{cases}
ey\epsilon(x) & \text{if }|y|=0\\
0&\text{otherwise}
\end{cases}
\]
is a semifree resolution of $M$ over $E$; the $E$-action is on the left $E$-factor of $U_E(F)$. 
\end{chunk}

\begin{chunk}
\label{c:dgstructure}
We use \cref{notation} and suppose $\bm{g} = g_1,\ldots,g_d$ is a sequence of elements in $Q$ with $(\f)\subseteq(\bm{g})$. 

Fix  a $Q$-semifree dg algebra resolution $A\xra{\simeq} Q/(\bm{g})$.
Writing
\begin{equation}
\label{e:multiplication1}
    f_i=\sum_{j=1}^d a_{ij} g_j \quad\text{with}\quad a_{ij}\in Q\,
\end{equation}
defines a morphism of dg $Q$-algebras $E\to A$ determined by
\begin{equation}
\label{e:multiplication2}
   e_i\mapsto \sum_{j=1}^d a_{ij} e_j'\quad\text{with}\quad e'_j\in A_1\,, \ \del e_j'=g_j\,. 
\end{equation}
In particular, if $(\f)\subseteq J (\bm{g})$ for some ideal $J$ in $Q$, then one can take the $a_{ij}$ in \cref{e:multiplication1} to belong to $J$ and hence, the morphism in \cref{e:multiplication2} defines a dg $E$-module structure on $A$ where 
the image of multiplication by $e_i$ on $A$  is contained in $JA.$
\end{chunk}

\begin{lemma}
\label{l:equality}
If $M$ is an $R$-complex with $I\subseteq \m\ann_Q(M)$, then there is the following isomorphism of graded $k$-spaces
\[
\Tor^E(M,k)\cong \Tor^Q(M,k)\otimes_Q \varGamma\,.
\] 
% In particular, $\P^E_M(t)=P^Q_M(t)\cdot (1-t^2)^{-n}\,.$
\end{lemma}
\begin{proof}
Let $A\xra{\simeq} k$ be a $Q$-semifree dg algebra resolution of $k$; see \cite[Section~6.3]{Avramov:2010}. From the assumption $I\subseteq \m\ann_Q(M)$, and applying \cref{c:dgstructure} with $\bm{g}$ a list of minimal generators for $\m$, it follows that the dg $E$-module structure on $A$ can be taken to satisfy the following for each $i$: 
\begin{equation}
\label{e:trivial_action}
e_iA\subseteq \ann_Q(M) A\,.
\end{equation}

Also, there is the following commutative diagram of graded $Q$-modules 
 \begin{equation}
    \begin{tikzcd}
    \label{comm_diagram_1}
    M\otimes_E U_E(A) \ar[r,"\cong"] \ar[d,swap,"1_M\otimes e_i\otimes \chi_i\otimes 1_A"] & M\otimes_Q\varGamma\otimes_Q A\ar[d,swap,"e_i\otimes\chi_i\otimes 1_A"'] \\
  M\otimes_E U_E(A) \ar[r,"\cong"] & M\otimes_Q\varGamma\otimes_Q A
    \end{tikzcd}
    \end{equation}
where the horizontal maps are induced by the multiplication map $M\otimes_E E\xra{\cong} M$ and the vertical maps have degree $-1.$ By \cref{restricting_e_action}, the right-hand map in \cref{comm_diagram_1} is zero and hence so is the left-hand map. Similarly, there is the following commutative diagram of graded $Q$-modules 
 \begin{equation}
    \begin{tikzcd}
    \label{comm_diagram_2}
    M\otimes_E U_E(A) \ar[r,"\cong"] \ar[d,swap,"1_M\otimes 1_E\otimes \chi_i\otimes e_i"] & M\otimes_Q\varGamma\otimes_Q A\ar[d,swap,"1_M\otimes\chi_i\otimes e_i"'] \\
  M\otimes_E U_E(A) \ar[r,"\cong"] & M\otimes_Q\varGamma\otimes_Q A
    \end{tikzcd}
    \end{equation}
    where the horizontal maps are again induced by $E\otimes_E E\xra{\cong} E$. This time the right-hand map in \cref{comm_diagram_2} is zero because of \cref{e:trivial_action}. In particular, the degree $-1$ maps 
    \[
    1_M\otimes e_i\otimes \chi_i\otimes 1_A\quad\text{and}\quad 1_M\otimes 1_E\otimes \chi_i\otimes e_i
    \]
    are both zero on $ M\otimes_E U_E(A)$. In view of the definition of the differential of $U_E(A)$ in \cref{c:kresolutions}, it follows that the isomorphism $M\otimes_E U_E(A)\cong M\otimes_Q\varGamma\otimes_Q A$  of graded $Q$-modules  is in fact one of complexes. Therefore, we have the following isomorphisms in homology: 
    \begin{align*}
        \Tor^E(M,k)&=\h(M\otimes_E U_E(A))\\
        &\cong\h(M\otimes_Q\varGamma\otimes_Q A) \\
        & \cong \h(M\otimes_Q A\otimes_Q \varGamma)\\
        &\cong \h(M\otimes_Q A)\otimes_Q\varGamma\\
        &=\Tor^Q(M,k)\otimes_Q  \varGamma\,;
    \end{align*}
the first and second equalities use that $U_E(A)$ and $A$ are semifree $E$- and $Q$-resolutions of $k$, respectively; the first isomorphism was what was justified above, while the second isomorphism is obvious, and the third isomorphism is because $\varGamma$ is a free graded $Q$-module.
\end{proof}

% \begin{proposition}
% \label{p:inequality}
% Let $R=Q/(\f)$ and $M$ an $R$-complex with finite homology.
% There is an equality of Poincar\'e series 
% \[
% \P_M^E(t)\preccurlyeq \P_M^R(t)\cdot P_R^E(t)\,.
% \]
% \end{proposition}

% \begin{proof}
% Consider the Cartan-Eilenberg spectral sequence 
% \[
% \E_{p,q}=\Tor^R_p(M,\Tor^E_q(R,k))\implies\Tor^E_{p+q}(M,k)
% \]

% \end{proof}

% \begin{proposition}
% Let $M$ be a dg-module with finite homology over $R$, Then 
% \[
% \ps_M^Q(t)\preccurlyeq \ps_M^E(t)\cdot P_E^Q(t)=\ps_M^E(t)\cdot (1+t)^n\preccurlyeq \ps_M^R(t)\cdot \ps_R^E(t)\cdot(1+t)^n\,.
% \]
% If $I\cap {\m_Q}^2\subseteq \m_QI$, then equality holds on the left, and, if furthermore $I$ is qci then equality holds on the right as well. 
% \end{proposition}\liana{Does this look right?}

\begin{proof}[Proof of \cref{t:Koszul-ineq-eq}]
We first prove the inequality \eqref{e:RQ2} and \cref{t_shamash_ci}. Let $F\xra{\simeq} M$ be a semifree resolution of $M$ over $E$. Observe that as graded $k$-spaces there are isomorphisms
\begin{equation}\label{e:iso}
k\otimes_E U_E(F)\cong k\otimes_Q \varGamma \otimes_Q F\cong (k\otimes_Q \varGamma)\otimes_k (k\otimes_Q F)\,.
\end{equation}
By \cref{c:kresolutions}, the homology of the left-hand side is $\Tor^E(M,k)$. Thus $\Tor^E(M,k)$ is   a subquotient of $k\otimes_EU_E(F)$ as a graded $k$-vector space, and the coefficient-wise inequality below has been justified:
\begin{align*}
\ps^E_M(t) &\preccurlyeq  \h_{k\otimes_E U_E(F)} \\
&= \h_{k\otimes_Q \varGamma}(t)\cdot  \h_{k\otimes_Q F}(t)\\
&=\frac{1}{(1-t^2)^n}\cdot \ps^Q_M(t)\,;
\end{align*}
the first equality holds using \cref{e:iso} and the second holds using the definition $\varGamma$ and that $F\to M$ is a semifree resolution over $Q$, since $E$ is free over $Q$. Using \cref{l:equality}, we see that the coefficient-wise inequality above is an equality when $I \subseteq \m\ann_Q(M)$.

We now prove \eqref{e:QR2} and \cref{t_nagata_ci }. Using induction we can assume $n=1$. For the inequality, fix a minimal semifree resolution $U$ of $M$ over $E$. Write 
   \begin{equation}
   \label{decomposition_U}
   U^\natural \cong (V \oplus  V e)^\natural 
   \end{equation}
   as graded $E^\natural\cong Q\oplus Qe$-modules; above $V$ denotes the $Q$-linear span of the semifree basis of $U$ as a dg $E$-module, thus it is a bounded below free graded $Q$-module.
Since $U$ is minimal over $E$ we have 
\begin{equation}
       \Tor^E(M,k)=U\otimes_E k\cong  V\otimes_Q k\,.
   \end{equation}
   Since $E$ is free over $Q$, it follows that $U$ is a free resolution of $M$ over $Q$. In particular, $\Tor^Q(M,k)$ is a  subquotient of $U\otimes_Q k$ regarded a graded $k$-vector space and hence
\begin{equation}\label{poincare_series_comparison}
   \P^Q_M(t)\preccurlyeq \h_{U\otimes_Q k}(t)\,.
   \end{equation}
   Also,    observe that there are isomorphisms of graded $k$-vector spaces
   \begin{align*}
   U\otimes_Q k
   &=(V\oplus  Ve)^\natural \otimes_Q k\\
   &\cong (V\otimes_Q k)\oplus (Ve\oplus_Q k)\\
   &\cong \Tor^E(M,k)\oplus \shift \Tor^E(M,k)
   \end{align*}
   and as a consequence $\h_{U\otimes_Q k}(t)=(1+t)\P^E_M(t)$. Combining this equality with the inequality from \cref{poincare_series_comparison} yields the desired inequality: 
   \[
   \P^Q_M(t)\preccurlyeq (1+t)\P^E_M(t)\,.
   \]

  Next we verify equality holds when $f\in \m\smallsetminus \m^2$. By  \cite[Proposition~2.2.2]{Avramov:2010}, the minimal free resolution $U$ of $M$ over $Q$ is a semifree dg $E$-module. As $U$ is minimal over $Q$ it is also minimal over $E$ and hence, we can write $U$ as in \cref{decomposition_U}. 
Therefore, equality holds in \cref{poincare_series_comparison} giving the desired equality. 
\end{proof}

\begin{remark}
    Note that the inequality \eqref{e:QR2} can also be justified using the spectral sequence in \ref{c:change_of_rings}. In fact, this spectral sequence degenerates when $N=k$ if and only if $Q\to E$ is large in the sense \cref{r:large}. Moreover, one can argue as in \cite{Levin:1980}, that when $I\cap \m^2 \subseteq \m I $ the spectral sequence degenerates and use this to give an alternate proof of the equality in \cref{poincare_series_comparison}.
\end{remark}

%%%%%%%%%%%%%%%%%%%%%%%%%%%%%%%%%%%%%%
\section{Quasi-complete intersection homomorphisms}
\label{sec:qci}
%%%%%%%%%%%%%%%%%%%%%%%%%%%%%%%%%%%%%%

In this section, $(Q,\m,k)$ is a local ring, $\varphi\colon Q\to R$ is a surjective homomorphism and we set $I=\Ker\varphi$. 

% \begin{chunk}
% \label{c:deviations}
% deviations \josh{Josh} \liana{Do we really want to define deviations in general? For what we need, we could define some notation for the cardinalities of $X_1$ and $X_2$ below, and maybe mention in passing the deviation terminology.}
% \end{chunk}

\begin{chunk}
\label{c:qci}{\bf Quasi-complete intersection homomorphisms.}
Let $\f=f_1, \dots, f_n$ be a minimal generating set of $I$, and let $E$ be the Koszul complex on $\f$. Following the procedure recalled in \ref{c:gamma}, construct the {\it two-step Tate complex}
\[
F\coloneq Q\langle X_1,X_2\rangle\,,
\] 
where $X_1$, $X_2$ are two sets of variables such that $Q\langle X_1\rangle =E$ and the variables in $X_2$ kill a basis of $\h_1(E)$. That is to say, the differential on $F$ maps $X_2$ bijectively to a set of cycles whose homology classes minimally generate $\h_1(E).$

The map $\varphi$ is said to be a {\it quasi-complete intersection} (q.c.i.) homomorphism if $\h_1(E)$ is a free $R$-module and the natural map 
\[
{\sf \Lambda}_R\h_1(E)\to \h(E)
\]
is an isomorphism of graded $R$-algebras. This property first appeared in work of Rodicio \cite{Rodicio:1995} and Blanco, Majadas and Rodicio \cite{Blanco/Majadas/Rodicio:1996}, and the current terminology was introduced by Avramov, Henriques and \c{S}ega \cite{Avramov/Henriques/Sega:2013}. According to \cite[Theorem~1]{Blanco/Majadas/Rodicio:1998}, $\varphi$ is q.c.i.~if and only if $F$ is the minimal free resolution of $R$ over $Q$. 

Such maps can also be defined in terms of vanishing of Andr\'e-Quillen functors  $\mathrm{D}_i(R/Q;-)$ whenever $i\ges 3$, and Quillen~\cite{Quillen:1970} conjectured these are the only maps with this kind of behavior: \emph{if $\mathrm{D}_i(R/Q;-)=0$ for  $i\gg 0$, then $\varphi$ must be q.c.i.} See \cite{Avramov/Henriques/Sega:2013} for more details regarding these homomorphisms.
\end{chunk}
\begin{chunk}
\label{equality}
Recall that $\grade_QR$ denotes the maximal length of a $Q$-regular sequence in $I$. 
Assume that $\varphi$ is a q.c.i.~map. By \cite[Lemma 1.2]{Avramov/Henriques/Sega:2013} and \cite[Theorem 4.1]{Avramov/Gasharov/Peeva:1997}, we have
\begin{equation}
\label{e:qci-grade}
\depth Q-\depth R=\grade_QR=|X_1|-|X_2|\,,
\end{equation}
where $F=Q\langle X_1,X_2\rangle$ is the two-step Tate resolution of $R$ over $Q$.
Also,  the following formula holds
\begin{equation}
\label{e:Poincare-formula}
    \P_N^R(t)\cdot\frac{(1-t)^{\edim R}}{(1-t^2)^{\depth R}}=\P_N^Q(t)\cdot\frac{(1-t)^{\edim Q}}{(1-t^2)^{\depth Q}}\,
\end{equation}
when $N=k$ by \cite[Theorem 6.1]{Avramov/Henriques/Sega:2013} and for any finitely generated $R$-module $N$ when $\m^2\cap I\subseteq \m I$ by \cite[Theorem~6.2]{Avramov/Henriques/Sega:2013}. The proof of \cref{e:Poincare-formula} in the later case is based on the fact,  established in the proof of \cite[Theorem~6.2]{Avramov/Henriques/Sega:2013}, that the homomorphism $\varphi$ is large. Our discussion in \cref{r:large} yields that \eqref{e:Poincare-formula} holds, more generally, for all $N$ in  $\Dfp(R)$ when $\m^2\cap I\subseteq \m I$.

Finally, observe that \[F=Q\langle X_1,X_2\rangle=E\langle X_2\rangle\] is the minimal semifree resolution of $R$, considered as a dg module over $E$, and thus 
\begin{equation}
\label{e:RoverE}
\P_R^E(t)=\frac{1}{(1-t^2)^{|X_2|}}\,.
\end{equation}
\end{chunk}

\begin{chunk}
If $M$ is in $\Dfp(R)$, then the following inequality holds: 
 \begin{equation}
 \label{e:inert}
\P^R_M(t)\P^Q_k(t)\preccurlyeq \P^Q_M(t)\P^R_k(t)\,.
 \end{equation}
This was proved by Lescot in \cite{Lescot:1990} in the case that $M$ is an $R$-module. The proof relies on a  convergent spectral sequence, that can be extended for $M$ in $\Dfp(R)$.  
 Following Lescot, when equality holds in \eqref{e:inert}, $M$ is said to be {\it inert} by $\varphi$.
 
 If $I$ is generated by a regular sequence of length $n$, then \ref{c:recall-ci}(1)  asserts that any object $M$ in $\Dfp(R)$ with $I\subseteq \m\ann_Q(M)$ is inert by $\varphi$. 
\end{chunk}

\begin{remark}\label{r:equivalent-inert}
If $\vp$ is q.c.i.\@ and $M$ is in $\Dfp(R)$, then the following are equivalent:
\begin{enumerate}
    \item $M$ is inert by $\vp$;
    \item there is an equality of formal power series \[ \P_M^R(t)\cdot\frac{(1-t)^{\edim R}}{(1-t^2)^{\depth R}}=\P_M^Q(t)\cdot\frac{(1-t)^{\edim Q}}{(1-t^2)^{\depth Q}}\,.\]
\end{enumerate}
Furthermore, if $I\subseteq \m \ann_Q(M)$, then the following is also equivalent:
\begin{enumerate}
    \item[(3)] $\P^R_M(t)=\P^Q_M(t)\cdot (1-t^2)^{\grade_Q R}\,.$
\end{enumerate}
Indeed, the equivalence of (1) and (2) is straightforward using the already noted fact that \eqref{e:Poincare-formula} holds with $N=k$. 

Next assume that $I\subseteq \m \ann_Q(M)$. If $\ann_Q(M)=Q$, then all of these equalities hold vacuously. So we can further assume $\ann_Q(M)\subseteq \m$, and hence $I\subseteq \m^2$; therefore,  $\edim Q=\edim R$. It now follows from a direct computation, using also \eqref{e:qci-grade}, that (2) and (3) are equivalent. 
\end{remark}

\begin{theorem}
\label{t:main}
Let $(Q,\m,k)$ be a local ring, $\vp\colon Q\to R$ a surjective quasi-complete intersection map, and  set $I=\Ker\vp$. For  $M$  in $\Dfp(R)$ with $I\subseteq \m\ann_Q(M)$, then 
\[
\P^R_M(t)=\P^Q_M(t)\cdot (1-t^2)^{\grade_Q R}\,.
% \dfrac{\P^Q_M(t)}{(1-t^2)^{\varepsilon_2(\vp)}}=\dfrac{\P^R_M(t)}{(1-t^2)^{\varepsilon_3(\vp)}}
\]
Equivalently, $M$ is inert by $\vp$.
\end{theorem}
\begin{proof}
Let $E$ denote the Koszul complex on a minimal generating of $I$, and let $F=R\langle X_1, X_2\rangle$ be the two-step Tate complex, as in \ref{c:qci}. Observe $|X_1|=n$ and set
\[
m\coloneq|X_2|=\rank_k (\h_1(E)\otimes_R k)\,.
\]
There is nothing to show when $\ann_Q(M)=Q$, so we can assume $\ann_Q(M)\subseteq \m$. It follows that $\edim R=\edim Q$, and since \eqref{e:Poincare-formula} holds with $N=k$ we can use \eqref{e:qci-grade} to obtain 
\begin{equation}
\label{e:qcik}
\P_k^Q(t)=\P_k^R(t)\cdot (1-t^2)^{n-m}\,.
\end{equation}

The following is justified by  \cref{t:Koszul-ineq-eq}\cref{t_shamash_ci}, \cref{l:large-ineq} (applied to the map of local dg algebras $E\to R$), and \eqref{e:RoverE}:
\begin{equation}
    \label{e:1}
    \dfrac{\P^Q_M(t)}{(1-t^2)^{n}}=P_M^E(t)\preccurlyeq \P_M^R(t)\cdot \P_R^E(t)=\frac{P_M^R(t)}{(1-t^2)^{m}}\,;
\end{equation}
Now observe that
\begin{align*}
\frac{\P_M^R(t)\cdot \P_k^Q(t)}{(1-t^2)^{m}}&\preccurlyeq \frac{\P_M^Q(t)\cdot \P_k^R(t)}{(1-t^2)^{m}}\\
&=\frac{\P_M^Q(t)}{(1-t^2)^{n}}\cdot\frac{\P_k^R(t)}{(1-t^2)^{m-n}}\\
&=\frac{\P_M^Q(t)}{(1-t^2)^{n}}\cdot\P_k^Q(t)\,;\end{align*}
the first coefficient-wise inequality is from \eqref{e:inert}, the first equality is clear, and the last equality is from \eqref{e:qcik}. From this and \cref{e:1}, it follows that 
\[
\frac{\P_M^R(t)\cdot \P_k^Q(t)}{(1-t^2)^{m}}=\frac{\P_M^Q(t)\cdot\P_k^Q(t)}{(1-t^2)^{n}}\,.
\]
Canceling the factors of $\P^Q_k(t)$, and another application of \eqref{e:qcik} yields the desired equality in the statement. By \cref{r:equivalent-inert},  this equality holds if and only if $M$ is inert by $\vp$. 
\end{proof}

\begin{remark}
    The proof of \cref{t:main} shows that, under the hypotheses of the theorem, equality must hold in \eqref{e:1}, and thus: 
    \[
    \P_M^E(t)=\frac{P_M^R(t)}{(1-t^2)^{m}}\,.
    \]
\end{remark}

\begin{remark}
\label{r:mingen}
If a minimal generating set of $I$ can be extended to a minimal generating set of $\m$, that is to say $I\cap \m^2\subseteq \m I$, then $\varphi$ is large, as noted in \cref{equality}.  As a consequence, factoring $\vp$ as
$Q\to E\to R$
it follows that $E\to R$ is large. 

We remark that one can directly show, essentially by the same argument in \cite[Theorem~5.3]{Avramov/Henriques/Sega:2013}, that $E\to R$ is large when $I\cap \m^2\subseteq \m I$ and combining this with \cref{t:Koszul-ineq-eq}\cref{t_nagata_ci } we recover that \cref{e:Poincare-formula} holds for any  $N$ in $\Dfp(R)$; this would go through the base change formula on Poincar\'e series for the Koszul extension $Q\to E$, and hence would be analogous to the proof of \cref{t:main}. 
\end{remark}

We end the paper with another coefficient-wise inequality comparing Poincar\'e series along surjective q.c.i.\@ homomorphisms, extending some previously known results; see \cref{r:previous_inequality,r:cx}. 

\begin{proposition}\label{e:generalineq}
Let $\varphi\colon Q\to R$ by a surjective quasi-complete intersection map. Let $E$ denote the Koszul complex on a set of minimal generators of $\Ker\varphi$ and set $n=\rank_Q(E_1)$ and $m=\rank_k (\h_1(E)\otimes_R k).$ For any $M$ in $\Dfp(R)$ we have a coefficient-wise inequality
\begin{equation*}
\P^Q_M(t)\preccurlyeq \frac{(1+t)^{n-m}}{(1-t)^m}\P_M^R(t)\,.
\end{equation*}
Equality holds above when $I\cap \m^2\subseteq \m I$.
\end{proposition}
\begin{proof}
    Putting together equations \cref{e:QR2} and \cref{e:RoverE}, and \cref{l:large-ineq}, we obtain the coefficient-wise (in)equalities 
\begin{align*}
\P^Q_M(t)&\preccurlyeq \P_M^E(t)\cdot(1+t)^n\preccurlyeq \P_M^R(t) \P_R^E(t)\cdot (1+t)^n=\P_M^R(t)\cdot \frac{1}{(1-t^2)^m}\cdot (1+t)^n\,,
\end{align*}
yielding the desired inequality. When $I\cap \m^2\subseteq \m I$, equality holding has already been noted in \cref{equality}.
\end{proof}

\begin{remark}\label{r:previous_inequality}
Let $m$, $n$ be as in \cref{e:generalineq}.   
    The inequality in \cref{e:generalineq} is an extension of \eqref{e:QR}, which covers the case $m=0$; cf.\@ the discussion at the end of \cref{c:recall-ci}. It also extends \cite[Corollary 3.6]{Bergh/Celikbas/Jorgensen:2014}, which addresses the case when $n=m=1$ (that is, when $\Ker\varphi$ is generated by an \emph{exact zero divisor}). 
\end{remark}
    
The inequality established in \cref{e:generalineq} can be used to relate asymptotic invariants of  $M$ along $\varphi$ as described below.

   Recall the  {\it complexity} and {\it curvature} of $M$ over $R$, denoted $\cx_R(M)$ and $\curv_R(M)$ respectively, measure the polynomial and the exponential rate of growth of the Betti sequence of $M$ over $R$, respectively. 
   See \cite[Section~4]{Avramov:2010} for precise definitions and more details. The following is an immediate consequence of \cref{e:generalineq}.

%    is defined as the least integer $c>0$
% such that, for $n\gg 0$, $\beta_n(M) < \gamma n^{c-1}$ for some real number $\gamma$. The curvature of $M$ over $R$, denoted $\curv_R(M)$, is the
% exponential rate of growth of the Betti sequence, defined as the reciprocal
% of the radius of convergence of the Poincar\'e series $\P^R_M(t)$.

\begin{corollary}
\label{cor:complexity}
In the notation of \cref{e:generalineq},  the following inequalities hold
\begin{align*} 
\begin{split}
\pushQED{\qed} 
    \cx_Q(M)&\les \cx_R(M) +m\\
    \curv_Q(M)&\les 
    \max\{\curv_{R}(M),1\}\,.  \qedhere
    \pushQED{\qed} 
    \end{split}
\end{align*}
\end{corollary}
 %When $m>0$, the following upper bounds for $\cx_Q(M)$ and $\curv_Q(M)$ follow directly from \eqref{e:generalineq}: 

\begin{remark}\label{r:cx}
Continuing with the notation from \cref{e:generalineq}, when $m=0$, (i.e. $\varphi$ is a complete intersection homomorphism) the inequalities in \cref{cor:complexity} are well known. In fact, in this case stronger relationships for these invariants over $Q$ and $R$  follow from the inequalities in \cref{c:recall-ci}. Namely, 
\begin{align*}
    &\cx_Q(M)\les\cx_R(M)\les \cx_Q(M) +n\\
    \curv_Q(M)&=\curv_{R}(M)\quad\text{when }\pdim_R(M)=\infty\,;
\end{align*}
see \cite[Proposition~4.2.5(4)]{Avramov:2010}.

    When $m>0$, as far as the authors are aware of, the only known results that give similar lower bounds for $\cx_Q(M)$ and $\curv_Q(M)$,  in terms of the invariants defined over $R$, are  established in recent joint work of the second author in \cite{Sega/Sireeshan:2024}. There $\Ker\varphi$ is generated by an exact zero divisor (that is, $n=m=1$) and the residue field has characteristic zero. We expect similar lower bounds to hold more generally, but we do not have additional evidence at this time. 
\end{remark}

\begin{ack} This material is based upon work supported by the National Science Foundation under Grant No.\@ DMS-1928930 and by the Alfred P. Sloan Foundation under grant G-2021-16778, while the second author was in residence at the Simons Laufer Mathematical Sciences Institute (formerly MSRI) in Berkeley, California, during the Spring 2024 semester. The first author was supported by the National Science Foundation under Grant No.\@ DMS-2302567.
Finally, we thank the referee for their comments that greatly improved the paper. 
\end{ack}

%%%%%%%%%%%%%%%%%%%%%%%%%%%%%%%%%%%%%%%%%%%%%%%%%%%%%%%%%%%
%%%%%%%%%%%%%%%%%%%%%%%%%%%%%%%%%%%%%%%%%%%%%%%%%%%%%%%%%%%
\bibliographystyle{amsplain}
\bibliography{refs}

\end{document}